\documentclass[reqno, 10pt, a4paper, oneside]{amsart}
\usepackage{amsmath}
\usepackage{amsfonts}
\usepackage{amssymb}
\usepackage[ansinew]{inputenc}
\usepackage{graphicx}
\usepackage{amsbsy}
\usepackage{fancyhdr}
\usepackage{yfonts}
\usepackage{hyperref}
\usepackage{eufrak}
\usepackage{enumerate}

\numberwithin{equation}{section}
\newtheorem{teo}{Theorem}[section]
\newtheorem{prop}[teo]{Proposition}
\newtheorem{lema}[teo]{Lemma}

\newtheorem{coro}[teo]{Corollary}
\theoremstyle{definition}
\newtheorem{defi}[teo]{Definition}

\newtheorem{ej}[teo]{Example}

\title{Norm convergence of nilpotent ergodic averages}
\author{Miguel N. Walsh}
\address{Departamento de Matemática, Facultad de Ciencias Exactas y Naturales, Universidad de Buenos Aires, 1428 Buenos Aires, Argentina}
\email{mwalsh@dm.uba.ar}
\thanks{The author was partially supported by a CONICET doctoral fellowship.}

\begin{document}
\maketitle

\begin{abstract}
We show that multiple polynomial ergodic averages arising from nilpotent groups of measure preserving transformations of a probability space always converge in the $L^2$ norm.
\end{abstract}

\bigskip

\section{Introduction}

The purpose of this paper is to prove the following result.

\begin{teo}
\label{original}
Let $G$ be a nilpotent group of measure preserving transformations of a probability space $(X, \mathcal{X}, \mu)$. Then, for every $T_1, \ldots, T_l \in G$, the averages
\begin{equation}
\label{general ergodic averages}
\frac{1}{N} \sum_{n=1}^{N} \prod_{j=1}^d  \left( T_1^{p_{1,j}(n)} \ldots T_l^{p_{l,j}(n)} \right) f_j
\end{equation}
always converge in $L^2(X, \mathcal{X}, \mu)$, for every $f_1, \ldots, f_d \in L^{\infty}(X, \mathcal{X}, \mu)$ and every set of integer valued polynomials $p_{i,j}$.
\end{teo}

This result was conjectured in the present form by Bergelson and Leibman, who also showed that even $\lim_{N \rightarrow \infty} \frac{1}{N} \sum_{n=1}^{N} T^n f S^n g$ need not exist if $T$ and $S$ only generate a solvable group \cite{BL}. 

\subsection{Historical background}

Partial results towards Theorem \ref{original} have a rich history. Notice that when $d=l=1$ and the polynomial is linear it reduces to the classical mean ergodic theorem. The only case of Theorem \ref{original} which was fully settled is that in which $T_1 = \ldots = T_l$, that is, when $G$ is a cyclic group. The study of this case originated in the seminal work of Furstenberg \cite{Fur} on Szemerédi's theorem, while a general solution when the polynomials are linear was later provided by Host and Kra \cite{HK} following the work of several authors (with a different proof subsequently found by Ziegler \cite{Ziegler}). Convergence for general polynomials was established by Bergelson \cite{B} under the assumption of weakly mixing, while the first unconditional non-linear result was obtained by Furstenberg and Weiss \cite{FW}. The general result for cyclic groups and arbitrary polynomials was finally settled by Host and Kra \cite{HK2} and Leibman \cite{L2}.

\medskip
Another case of Theorem \ref{original} which is known is that in which $G$ is abelian and every polynomial is linear. Here, the case $d=2$ was proven by Conze and Lesigne \cite{CL} and assuming extra ergodicity hypothesis on the transformations Zhang \cite{Zhang} gave a proof for $d=3$ and Frantzikinakis and Kra \cite{FK} for general $d$. Without these assumptions, this result was established by Tao \cite{Tao} and by now possesses several different proofs \cite{Austin,Host,Towsner}. However, when $G$ is abelian but the polynomials are arbitrary, very little was known. It was shown by Chu, Frantzikinakis and Host \cite{CFH} that
\begin{equation}
\label{one poly}
 \frac{1}{N} \sum_{n=1}^N T_1^{p_1(n)} f_1 \ldots T_l^{p_l(n)} f_l
 \end{equation}
converges whenever the polynomials $p_i$ have distinct degrees, but the convergence of (\ref{one poly}) has remained open for arbitrary polynomials. Notice that (\ref{one poly}) corresponds to taking $p_{i,j}=0$ whenever $i \neq j$ in Theorem \ref{original}. More generally, very little was known until now for convergence of $\mathbb{Z}^d$ actions along polynomials. A particular result in this direction is the convergence of the averages
$$  \frac{1}{N} \sum_{n=1}^N T^{n^2} f \left( T^{n^2}S^n \right) g,$$
which was established by Austin \cite{A1,A2}.

\medskip
Finally, when $G$ is only assumed to be nilpotent the results are much scarcer. Prior to this paper, it was known by the work of Bergelson and Leibman \cite{BL} that the averages
$$ \frac{1}{N} \sum_{n=1}^N T^n f S^n g,$$
always converge in $L^2$, but even in the linear case no convergence result has been previously established for more than two transformations. 

\subsection{Overview of the proof}

Our proof of Theorem \ref{original} does not make use of the aforementioned results and therefore provides an alternative proof of these statements, which in many cases is substantially simpler than the original ones. In particular, we do not make use of the machinery of characteristic factors which is heavily used in previous literature. The price we pay in doing so is that we do not obtain any explicit description of the limits. In this sense, our approach is similar to that of Tao \cite{Tao}, in that we use a weak inverse theorem (see Lemma \ref{inverso}) to decompose our functions into the sum of a random component, which is easily treatable, and a structured one, which can be handled by an adequate induction. Interestingly, we find that our decomposition is best carried out by adapting ideas of Gowers related to the Hahn-Banach separation theorem \cite{Gowers} and this is done in \S \ref{lemmata}. This is arguably the first time that these ideas are used in a purely ergodic theoretical context.

\medskip
The main new ingredient of the proof is the concept of an $L$-reducible function (Definition \ref{symmetric}), which will play the role of the structured component of our decompositions. We refer to \S \ref{main} for precise definitions, but for now let us discuss what these are in the linear abelian case. Here, an $L$-reducible function $\sigma$ with respect to a set of transformations $T_1, \ldots ,T_j$, is a function for which the behavior of $T_j^n \sigma$ can be somewhat recovered from that of the set $T_1^n b_1, \ldots, T_{j-1}^n b_{j-1}$, for some prescribed set of functions $b_i$. This way, the problem of convergence for the set of transformations $T_1, \ldots, T_j$ is reduced to the analogous question for the smaller set $T_1, \ldots, T_{j-1}$, and one may then proceed inductively. The details of these reductions are carried out in \S \ref{main}.

\medskip
When either $G$ is not abelian or the polynomials are not linear, the system of transformations to which $L$-reducible functions allow us to pass does not admit such a simple expression. In general, it will consist of twice as many transformations as the original one and the degree of the polynomials involved may not necessarily decrease, so that it may seem that we have not gained much with this procedure. As it turns out however, one can define a suitable notion of complexity for every set of transformations and show that the above process does indeed lead us to a set of lower complexity. The proof that every system of transformations of the type studied in Theorem \ref{original} reduces in finitely many steps to one consisting only of the identity transformation $Ix=x$ is performed in \S \ref{nilpotent}, and this completes the proof of Theorem \ref{original}.

\medskip
The methods of this paper immediately shield some further convergence results and these are discussed in Section \S \ref{further}. We also include an appendix with several examples of how the induction process mentioned in the previous paragraph works in some concrete cases.

\medskip
\noindent{\sl Acknowledgments.} I would like to thank my advisor Román Sasyk for a careful reading of the manuscript and several helpful discussions. I would also like to thank Tim Austin, Vitaly Bergelson and Nikos Frantzikinakis for their useful comments.

\section{Decompositions through the Hahn-Banach theorem}
\label{lemmata}

In this section we will review some of the tools developed by Gowers and adapt them to the context of our problem. For a much better discussion of these topics the reader is referred to Gowers's paper \cite{Gowers}. We begin by stating the Hahn-Banach separation theorem in the form it will be needed.

\begin{teo}[Geometric Hahn-Banach]
\label{HB}
Let $A$ be an open convex subset containing $0$ of a real topological vector space $V$ and suppose $v \in V$ does not lie in $A$. Then there exists some continuous linear functional $\phi: V \rightarrow \mathbb{R}$ such that $\phi(v) \ge 1$ and $\phi(w) < 1$ for every $w \in A$.
\end{teo}

The idea of Gowers to obtain decompositions can roughly be described as follows. While it may be difficult to check directly whether an arbitrary function can be described by the sum of a structured and a random component, if such a decomposition fails to exist an application of the Hahn-Banach theorem would allow us to find, since these sets tend to be convex, some large functional which does not correlate with random functions (therefore having a kind of structure itself) nor with structured functions (therefore also having some randomness). This way, we are only left with proving that no object can be random and structured at the same time, which generally tends to be an easier task.

\medskip
We will concentrate on the study of a real Hilbert space $\mathcal{H}$ with norm $\left\| \cdot \right\|$. In order to apply the above scheme, we will need the following corollary.

\begin{coro}[cf. {\cite[Corollary 3.2]{Gowers}}]
\label{HB2}
Let $A_1, \ldots, A_n$ be open convex subsets containing $0$ of some real Hilbert space $\mathcal{H}$. Let $c_1,\ldots,c_n>0$ be positive real numbers and suppose $f \in \mathcal{H}$ cannot be written as $\sum_{j=1}^n c_j f_j$ with $f_j \in A_j$. Then there exists some $\phi \in \mathcal{H}$ such that $\langle \phi, f \rangle \ge 1$ and $\langle \phi, g_i \rangle < c_i^{-1}$ for every $g_i \in A_i$.
\end{coro}

\begin{proof}
Since the set $A := \sum_{i=1}^n c_i A_i$ will be an open convex set in $\mathcal{H}$ containing $0$ but not $f$, it follows by the Hahn-Banach theorem that there exists some $\phi \in \mathcal{H}$ satisfying $\langle \phi, f \rangle \ge 1$ and $\langle \phi, g \rangle < 1$, for every $g \in A$. The result follows immediately, since $c_i g_i \in A$ for every $g \in A_i$.
\end{proof}

Given a positive real number $\delta$ and some decreasing function $\eta: \mathbb{R}^+ \rightarrow \mathbb{R}^+$ we will consider the sequence $C_1^{\delta, \eta}, \ldots, C_{\lceil 2\delta^{-2} \rceil}^{\delta, \eta}$ defined recursively by 
\begin{equation}
\label{central sequence}
 C_{\lceil 2\delta^{-2} \rceil}^{\delta, \eta} :=1 , C_{n-1}^{\delta, \eta} := \max \left\{ C_n^{\delta, \eta}, 2 \eta(C_n^{\delta, \eta})^{-1} \right\}.
 \end{equation}
We shall also write $C^{\delta, \eta}:=C_1^{\delta,\eta}$. These constants will provide the parameters for the decomposition obtained below and the fact that they are independent of the specific decomposition will allow us to do {\sl a priori} modifications on our set of structured functions so that they are better suited to the resulting bounds.

\medskip
Given some norm $\left\| \cdot \right\|_N$ on $\mathcal{H}$ equivalent to $\left\| \cdot \right\|$, we define its dual norm by
$$ \left\| f \right\|_N^{\ast} := \sup_{\left\| g \right\|_N \le 1} \left| \langle f, g \rangle \right|.$$
Notice that $\left\| \cdot \right\|_N^{\ast}$ is then also equivalent to $\left\| \cdot \right\|$. We will be concerned with the study of an infinite family of norms $\left( \left\| \cdot \right\|_N \right)_{N \in \mathbb{N}}$ measuring increasing rates of structure and for which their dual norms $ \left( \left\| \cdot \right\|_N^{\ast} \right)_{N \in \mathbb{N}}$ measure decreasing rates of randomness. As it turns out, we will need to work with this large family of norms {\sl simultaneously}, so that if we know one of the components is random at a level $A$ (that is, $\left\| \right\|_{A}^{\ast}$ is small), we need the other component to be structured at a much higher level $B$ (that is, $\left\| \right\|_{B}$ must be small for some $B$ much larger than $A$). This is accomplished by the following result.

\begin{prop}
\label{decomposition}
Let $\left( \left\| \cdot \right\|_N \right)_{N \in \mathbb{N}}$ be a family of norms on $\mathcal{H}$ equivalent to $\left\| \cdot \right\|$ and satisfying $\left\| \cdot \right\|_{N+1}^{\ast} \le \left\| \cdot \right\|_{N}^{\ast}$ for every $N$. Let $0 < \delta, c < 1$ be positive real numbers, $\eta: \mathbb{R}^+ \rightarrow \mathbb{R}^+$ some decreasing function and $\psi:\mathbb{N} \rightarrow \mathbb{N}$ some function satisfying $\psi(N) \ge N$ for all $N$. Then, for every integer $M_{\bullet}>0$, there exists a sequence 
$$ M_{\bullet} \le M_1 \le \ldots \le M_{\lceil 2 \delta^{-2} \rceil} \le M^{\bullet}= O_{M_{\bullet},\delta, c, \psi}(1), $$
which does not depend on the specific family of norms, with the property that for any $f \in \mathcal{H}$ with $\| f \| \le 1$, we can find some $1 \le i \le \lceil 2 \delta^{-2} \rceil$ and integers $A,B$ with $M_{\bullet} \le A < c M_i< \psi(M_i) \le B,$ such that we have the decomposition $f=f_1+f_2+f_3$ with 
$$ \| f_1 \|_{B} < C_i^{\delta, \eta}, \| f_2 \|_{A}^{\ast} < \eta(C_i^{\delta, \eta}), \| f_3 \| < \delta.$$ 
\end{prop}

\begin{proof} Our proof is modeled on the proof of Proposition 3.5 of \cite{Gowers}. Set $A_1 :=M_{\bullet}$, $M_1 :=\lceil c^{-1}A_1+1 \rceil$ and $B_1 := \psi(M_1)$. If there is no decomposition of the desired form with these parameters and $C_1 := C_1^{\delta, \eta}$ we may apply Corollary \ref{HB2} to obtain some $\phi_1 \in \mathcal{H}$ such that $\langle \phi_1 ,f \rangle \ge 1$, $\| \phi_1 \|_{B_1}^{\ast} \le C_1^{-1}$, $\| \phi_1 \|_{A_1}^{\ast \ast} \le \eta(C_1)^{-1}$ and $\| \phi_1 \| \le \delta^{-1}$, where we are using the fact that if $\left\| \cdot \right\|_N$ is some norm equivalent to $\left\| \cdot \right\|$, then $\left\{ f \in \mathcal{H} : \left\| f \right\|_N<1 \right\}$ is an open convex set in $\mathcal{H}$ containing $0$.

\medskip
Recursively, if we cannot find a decomposition with parameters $A_{j-1}$, $M_{j-1}$, $B_{j-1}$, $C_{j-1}$ we set $A_j:=B_{j-1}$, $M_j:=\lceil c^{-1} A_j +1 \rceil$, $B_j:=\psi(M_j)$ and $C_j:=C_j^{\delta, \eta}$. If no such decomposition exists with these parameters we can then use Corollary \ref{HB2} to find some $\phi_j \in \mathcal{H}$ with properties analogous to the ones above. This way we construct a sequence of elements obeying the orthogonality relationships
\begin{align*}
 \left| \langle \phi_j , \phi_i \rangle \right| &\le \| \phi_j \|_{A_j}^{\ast \ast} \| \phi_i \|_{A_j}^{\ast} \le \| \phi_j \|_{A_j}^{\ast \ast} \| \phi_i \|_{B_i}^{\ast} \\
 &\le \eta (C_j)^{-1} C_i^{-1} \le 1/2 ,
 \end{align*}
whenever $i<j$, by construction of $C_k$. But then, by the bounds on $\left\| \phi_i \right\|$, we obtain upon expanding the inner product 
 \begin{equation}
 \label{energy}
 \left\| \phi_1 + \ldots + \phi_r \right\|^2 \le \delta^{-2} r + \frac{r^2-r}{2},
 \end{equation}
 for each $r \le \lceil 2 \delta^{-2} \rceil$. On the other hand, the condition $\langle \phi_i , f \rangle \ge 1$ for all $i$ implies that the left-hand side of (\ref{energy}) is at least $r^2$. Since this is absurd for $r=\lceil 2 \delta^{-2} \rceil$ the result follows.
 \end{proof}

Finally, we also prove the following lemma that will be needed later.

\begin{lema}[cf. {\cite[Corollary 3.5]{Gowers}}]
\label{normas}
Let $\Sigma \subseteq \mathcal{H}$ be a bounded set and suppose the norm
\begin{equation}
\label{estructura}
\left\| f \right\|_{\Sigma} := \inf \left\{ \sum_{j=0}^{k-1} |\lambda_j| : f = \sum_{j=0}^{k-1} \lambda_j \sigma_j, \sigma_j \in \Sigma \right\},
\end{equation}
is well defined and equivalent to $\left\| \cdot \right\|$. Then its dual norm is given by $\left\| f \right\|_{\Sigma}^{\ast}=\sup_{\sigma \in \Sigma} \left| \langle f,\sigma \rangle \right|$.
\end{lema}

\begin{proof}
Given some $f \in \mathcal{H}$ it is clear on one hand that $$\sup_{\sigma \in \Sigma} |\langle f, \sigma \rangle| \le \sup_{\left\| g \right\|_{\Sigma} \le 1} | \langle f,g \rangle |.$$ On the other hand, for every $\epsilon > 0$, if $g= \sum_{j=0}^{k-1} \lambda_j \sigma_j$ with $\sum_{j=0}^{k-1} |\lambda_j| < 1+\epsilon$, then $| \langle f, g \rangle | \le (1+\epsilon) \sup_{\sigma \in \Sigma} |\langle f,\sigma \rangle|$. The result follows.
\end{proof}

\section{Norm convergence for systems of finite complexity}
\label{main}

From now on fix a nilpotent group $G$ and a probability space $X$ as in the statement of Theorem \ref{original}. By a $G$-sequence we shall mean a sequence $\left\{ g(n) \right\}_{n \in \mathbb{Z}}$ taking values in $G$. An ordered tuple ${\bf g} = (g_1, \ldots, g_j)$ of $G$-sequences will be called a system, and for each system one can ask whether the corresponding ergodic averages
\begin{equation}
\label{system averages}
\mathcal{A}_N^{\bf g} [f_1, \ldots, f_j] := \mathbb{E}_{n \in [N]} \prod_{i=1}^j g_i(n) f_i,
\end{equation}
converge in $L^2(X)$ for every $f_1, \ldots, f_j \in L^{\infty}(X)$. Here, for a finite set $A$ we write $\mathbb{E}_{x \in A} f(x) := \frac{1}{|A|} \sum_{x \in A} f(x)$ and for every positive integer $N$ it is $[N] := \left\{ 1, \ldots, N \right\}$. We say two systems are equivalent if they consist of the same $G$-sequences, so for example if $g,h$ are $G$-sequences then the system $(h,g)$ is equivalent to the system $(g,h)$, and so is $(g,h,h)$. Clearly, the convergence of the averages of the form (\ref{system averages}) for some system implies the convergence of the averages associated to every equivalent system, since $T(f_1) T(f_2)= T(f_1 f_2)$ for every $T \in G$ and $f_1, f_2 \in L^{\infty}(X)$.

\medskip
To each pair of $G$-sequences $g,h$ we will associate, for each positive integer $m$, the $G$-sequence
$$ \langle g | h \rangle_m (n) := g(n) g(n+m)^{-1} h(n+m),$$
and we define the $m$-reduction of a system ${\bf g}=(g_1, \ldots, g_j)$ to be the system
$$ {\bf g}_m^{\ast} = (g_1, \ldots, g_{j-1}, \langle g_j | 1_G \rangle_m, \langle g_j | g_1 \rangle_m, \ldots, \langle g_j | g_{j-1} \rangle_m),$$
where by a slight abuse of notation we write $1_G$ for the $G$-sequence $1_G(n):=1_G$, where $1_G$ is the identity of $G$. The main purpose of this section will be to show that one can deduce the convergence of the averages (\ref{system averages}) for some system ${\bf g}$ from knowing this (actually, the slightly stronger Theorem \ref{principal} below) for every reduction ${\bf g}_m^{\ast}$ of ${\bf g}$. This leads us to define the complexity of a system.

\begin{defi}[Complexity of a system]
We say a system ${\bf g}$ has complexity $0$ if it is equivalent to the trivial system $(1_G)$ (that is, the system consisting only on the sequence $1_G$). Recursively, we say a system ${\bf g}$ has complexity $d$, for some positive integer $d \ge 1$, if it is not of complexity $d'$ for any $0 \le d' < d$ and it is equivalent to some system ${\bf h}$ for which every reduction ${\bf h}_m^{\ast}$ has complexity $\le d-1$. We say a system has finite complexity if it has complexity $d$ for some integer $d \ge 0$.
\end{defi}

Given a system ${\bf g}=(g_1, \ldots, g_j)$, some set of functions $f_1, \ldots, f_j \in L^{\infty}(X)$ and a pair of integers $N,N'$, write
$$ \mathcal{A}_{N,N'}^{\bf g}[f_1, \ldots, f_j] := \mathcal{A}_{N'}^{\bf g}[f_1, \ldots, f_j] - \mathcal{A}_{N}^{\bf g}[f_1, \ldots, f_j].$$
We have the following result.

\begin{teo}
\label{principal}
Let $G$ and $X$ be as above and let $d \ge 0$. Let $F: \mathbb{N} \rightarrow \mathbb{N}$ be some non-decreasing function with $F(N) \ge N$ for all $N$ and let $\varepsilon > 0$ be some positive real number. Then, for every integer $M > 0$, there exists a sequence of integers 
\begin{equation}
\label{sequence0}
M \le M_1^{\varepsilon,F,d} \le \ldots \le M_{K_{\varepsilon, d}}^{\varepsilon,F,d} \le M^{\varepsilon, F,d}=O_{d, F,\varepsilon, M}(1),
\end{equation}
for some $K_{\varepsilon,d} = O_{\varepsilon,d}(1)$, such that for every system ${\bf g}=(g_1,\ldots,g_j)$ of complexity at most $d$ and every choice of functions $f_1, \ldots, f_j \in L^{\infty}(X)$ with $\left\| f_i \right\|_{\infty} \le 1$, there exists some $1 \le i \le K_{\varepsilon,d}$ such that
\begin{equation}
\label{finite average}
\left\| \mathcal{A}_{N,N'}^{\bf g} [f_1, \ldots, f_j] \right\|_{L^2(X)} \le \varepsilon,
\end{equation}
for every $M_i^{\varepsilon, F,d} \le N,N' \le F(M_i^{\varepsilon,F,d})$.
\end{teo}

This type of statement already appears in the works of Tao \cite{Tao} and of Avigad, Gerhardy and Towsner \cite{AGT}. Clearly, Theorem \ref{principal} implies that the averages (\ref{system averages}) converge in $L^2(X)$ for every system ${\bf g}$ of finite complexity, since otherwise one could find some $\varepsilon > 0$ and some increasing function $F: \mathbb{N} \rightarrow \mathbb{N}$ such that $\left\| \mathcal{A}_{N,F(N)}^{\bf g} [f_1, \ldots, f_d] \right\|_{L^2(X)} > \varepsilon$ for every integer $N$. The usefulness of Theorem \ref{principal} lies on its uniformity over all systems of a fixed complexity, which plays an important role in the inductive argument. In fact, the ergodic averages (\ref{system averages}) associated to a system ${\bf g}$ for which the reductions ${\bf g}_m^{\ast}$ do not satisfy stability bounds which are uniform on $m$ may not necessarily converge, even if the ergodic averages associated to each individual reduction ${\bf g}_m^{\ast}$ do converge. 

\medskip
The rest of this section is devoted to the proof of Theorem \ref{principal}. In \S \ref{nilpotent} we will show that every system of the form given in Theorem \ref{original} has finite complexity, thereby completing the proof of that theorem.

\subsection{$L$-reducible functions}
\label{reducible}

Since Theorem \ref{principal} is trivially true when $d=0$, we may proceed by induction. Thus, let $d > 0$ be some positive integer and assume the result holds for every $d' < d$. Let $F$ and $0 < \varepsilon < 1$ be as in the statement of the theorem and let ${\bf g}=(g_1,\ldots, g_j)$ be some system of complexity at most $d$. Since it clearly sufficies to prove the result for any system equivalent to ${\bf g}$, by definition of the complexity we may assume without lost of generality that ${\bf g}_m^{\ast}$ has complexity $\le d-1$ for every positive integer $m$.

\medskip
Let $C^{\ast}$ denote the quantity $C^{\delta,\eta}$ defined in (\ref{central sequence}) associated to $\delta := \varepsilon/(2^5 3)$ and $\eta(x):=\varepsilon^2/(2^3 3^3 x)$, so that in particular $C^{\ast}$ depends only on $\varepsilon$. We will sometimes use the shorthands $\left\| \cdot \right\|_{\infty}$ for $\left\| \cdot \right\|_{L^{\infty}(X)}$, $\left\| \cdot \right\|_2$ for $\left\| \cdot \right\|_{L^2(X)}$ and $\langle \cdot , \cdot \rangle$ for $\langle \cdot , \cdot \rangle_{L^2(X)}$. The following definition will be crucial.

\begin{defi}[reducible functions]
\label{symmetric}
Given a positive integer $L$, we say $\sigma \in L^{\infty}(X)$, $\left\| \sigma \right\|_{\infty} \le 1$, is an $L$-reducible function (with respect to ${\bf g}$), if there exists some integer $M > 0$ and a family $b_0,b_1,\ldots,b_{j-1} \in L^{\infty}(X)$ with $\left\| b_i \right\|_{\infty} \le 1$, such that for every positive integer $l \le L$ 
$$\left\| g_j(l)\sigma- \mathbb{E}_{m \in [M]}  \left( \langle g_j | 1_G \rangle_m(l) \right)b_0 \prod_{i=1}^{j-1} \left( \langle g_j|g_i \rangle_m(l) \right) b_i \right\|_{L^{\infty}(X)} < \frac{\varepsilon}{16C^{\ast}}.$$
\end{defi}

Reducible functions will play a similar role than the one played by basic anti-uniform functions in \cite{Tao}. We stress that we do not care for the value of $M$ in Definition \ref{symmetric}. We will show in Lemma \ref{inverso} below that every function giving rise to a large average must resemble a reducible function. The main feature of these objects is that the role of the $G$-sequence $g_j$ on the averages (\ref{system averages}) can essentially be recovered by means of the set of $G$-sequences $\langle g_j | 1_G \rangle_m, \langle g_j | g_1 \rangle_m, \ldots, \langle g_j | g_{j-1} \rangle_m$. 

\begin{lema}[Weak inverse result for ergodic averages]
\label{inverso}
Assume the inequality $\left\| \mathcal{A}_N^{\bf g} [f_1, \ldots, f_{j-1},u] \right\|_2 > \varepsilon/6$ holds for some $\left\| u \right\|_{L^{\infty}(X)} \le 3C$, some $1 \le C \le C^{\ast}$ and some $f_1, \ldots, f_{j-1} \in L^{\infty}(X)$ with $\left\| f_i \right\|_{\infty} \le 1$. Then, there exists some constant $0 < c_1 <1$, depending only on $\varepsilon$, such that for every positive integer $L < c_1 N$ there is an $L$-reducible function $\sigma$ with $\langle u, \sigma \rangle > 2\eta(C)$.
\end{lema}

\begin{proof}
We begin by noticing that $\left\| \mathcal{A}_N^{\bf g} [f_1, \ldots, f_{j-1},u] \right\|_2^2 = \langle u, h \rangle$, where
\begin{equation}
\label{shift1}
 h:= \mathbb{E}_{n \in [N]} g_j(n)^{-1} \mathcal{A}_N^{\bf g} [f_1, \ldots, f_{j-1},u] \prod_{i=1}^{j-1} g_j(n)^{-1} g_i(n) f_i.
 \end{equation}
We claim $\sigma := h/3C$ is an $L$-reducible function for every $L<c_1 N$ and some $0 < c_1 < 1$ depending only on $\varepsilon$, from where the result immediately follows since by the observation above it is $\langle u,\sigma \rangle > 2 \eta(C)$. 

\medskip
It remains to prove this claim. Write $c_1:=\frac{\varepsilon}{96 (C^{\ast})^{2}}$ and assume $0 < l < c_1N$. Then, if we shift $[N]$ to $[N]+l$ we see that the right hand side of (\ref{shift1}) changes by a magnitude of at most $6lC^{\ast}/N < \varepsilon/(16 C^{\ast})$ (since $\left\| \mathcal{A}_N^{\bf g} [f_1, \ldots, f_{j-1},u] \right\|_{\infty} \le 3C \le 3C^{\ast}$) and thus
$$ \left\| h -\mathbb{E}_{n \in [N]} g_j(l+n)^{-1} \mathcal{A}_N^{\bf g} [f_1, \ldots, f_{j-1},u] \prod_{i=1}^{j-1} g_j(l+n)^{-1} g_i(l+n) f_i \right\|_{L^{\infty}(X)} < \frac{\varepsilon}{16 C^{\ast}}.$$
Applying $g_j(l)$ we get
$$ \left\| g_j(l)h -\mathbb{E}_{n \in [N]} \left( \langle g_j | 1_G \rangle_n(l) \right) \mathcal{A}_N^{\bf g} [f_1, \ldots, f_{j-1},u] \prod_{i=1}^{j-1} \left( \langle g_j | g_i \rangle_n(l) \right) f_i \right\|_{L^{\infty}(X)} < \frac{\varepsilon}{16 C^{\ast}}.$$
The claim then follows with $M :=N$, $b_0 :=  \frac{1}{3C} \mathcal{A}_N[f_1, \ldots, f_{j-1},u]$ and $b_i := f_i$.
\end{proof}

As mentioned early, the advantage of $L$-reducible functions is that they allow us to reduce the study of the ergodic averages of ${\bf g}$ to the study of averages arising from the reductions ${\bf g}_m^{\ast}$, which we already know to satisfy uniform stability bounds by the induction hypothesis. This idea is carried out in the next proposition.

\begin{prop}
\label{low dim}
For every positive integer $M_{\ast}$ there exists a sequence 
\begin{equation}
\label{sequence}
M_{\ast} \le M_1 \le \ldots \le  M_{\tilde{K}} \le M^{\ast} = O_{M_{\ast},\varepsilon,d,F}(1),
\end{equation}
depending only on $M_{\ast}, \varepsilon, d$ and $F$, and with $\tilde{K}$ depending only on $\varepsilon$ and $d$, such that if $f_1, \ldots, f_{j-1} \in L^{\infty}(X)$ with $\left\| f_i \right\|_{\infty} \le 1$ and $f= \sum_{t=0}^{k-1} \lambda_t \sigma_t$, where $\sum_{t=0}^{k-1} |\lambda_t| \le C^{\ast}$ and each $\sigma_t$ is an $L$-reducible function for some $L \ge F(M^{\ast})$, then there exists some $1 \le i \le \tilde{K}$ such that 
$$ \left\| \mathcal{A}_{N,N'}^{\bf g} [f_1, \ldots, f_{j-1}, f] \right\|_{L^2(X)} \le \varepsilon/4,$$
for every pair $M_i \le N, N' \le F(M_i)$.
\end{prop}

\begin{proof}
For every $\sigma_t$ let $M^{(t)}$ be the integer coming from the definition of an $L$-reducible function and let $b_i^{(t)} \in L^{\infty}(X)$ be the corresponding family of functions. It follows from the definition of an $L$-reducible function that for every $N \le L$ and every $0 \le t \le k-1$ we may replace $\mathcal{A}_N^{\bf g}[f_1,\ldots, f_{j-1}, \sigma_t]$ by
$$ \mathbb{E}_{m \in [M^{(t)}]} \left[ \mathbb{E}_{n \in [N]} \left( \prod_{i=1}^{j-1}g_i(n) f_i \right)\left( \left( \langle g_j | 1_G \rangle_m(n) \right)b_0^{(t)} \right) \left( \prod_{i=1}^{j-1} \left( \langle g_j | g_i \rangle_m(n) \right)b_i^{(t)} \right) \right] ,$$
at the cost of an $L^{\infty}$ error of at most $\varepsilon/(16 C^{\ast})$. Therefore, we get by Minkowski's inequality that for $N,N' \le L$, $\left\|\mathcal{A}_{N,N'}^{\bf g}[f_1, \ldots, f_{j-1}, f] \right\|_{2}$ is bounded by
\begin{equation}
\label{minkowski}
\left( \sum_{t=0}^{k-1} |\lambda_t| \mathbb{E}_{m \in [M^{(t)}]} \left\|\mathcal{A}_{N,N'}^{{\bf g}_m^{\ast}}\left[f_1, \ldots, f_{j-1},b_0^{(t)},b_1^{(t)},\ldots,b_{j-1}^{(t)} \right] \right\|_{L^2(X)} \right)+\varepsilon/8.
\end{equation}

We are thus given a large family of averages coming from the lower complexity systems ${\bf g}_m^{\ast}$. Write $\gamma := \frac{\varepsilon}{16 C^{\ast}}$. Clearly, it would suffice to find a suitable interval on which each of this lower dimensional averages is bounded by $\gamma$. Although this will not be possible, we will indeed show by repeated applications of the induction hypothesis that we can get such a bound for all but a negligible subset of these averages. In order to do this, consider non-decreasing functions $F_1, \ldots, F_r:\mathbb{N} \rightarrow \mathbb{N}$, for some $r=O_{\varepsilon,d}(1)$ to be specified, defined recursively by $F_r := F$ and $F_{i-1}(N) := \max_{1 \le M \le N} F_{i}(M^{\gamma,F_{i},d-1})$, where we are using the notation in the statement of Theorem \ref{principal}. Also, let $K := K_{\gamma,d-1}$ be as in that theorem and for each tuple  $1 \le i_1, \ldots, i_s \le K$, $s \le r$, and integer $M$, we define recursively
$$ M^{(i_1, \ldots, i_s)} :=\left( \left( \left( M_{i_1}^{\gamma, F_1,d-1} \right)_{i_2}^{\gamma, F_2,d-1} \right) \ldots \right)_{i_s}^{\gamma, F_s,d-1}.$$
Thus, $M^{(i_1)}$ is the integer $M_{i_1}^{\gamma, F_1,d-1}$ obtained in (\ref{sequence0}) by starting at $M$, $M^{(i_1,i_2)}$ is the integer $M_{i_2}^{\gamma, F_2,d-1}$ obtained by starting the sequence (\ref{sequence0}) at $M=M^{(i_1)}$, etc. In particular, notice that this sequence depends only on $\varepsilon, F, d$ and $M$. Observe also that since each of the averages in (\ref{minkowski}) satisfies
\begin{equation}
\label{contribution}
 \left\|\mathcal{A}_{N,N'}^{{\bf g}_m^{\ast}}\left[f_1, \ldots, f_{j-1},b_0^{(t)},b_1^{(t)},\ldots,b_{j-1}^{(t)} \right] \right\|_{L^{\infty}(X)} \le 1,
 \end{equation}
the sum on (\ref{minkowski}) is bounded by $\sum_{t=0}^{k-1}|\lambda_t| \le C^{\ast}$.

\medskip
We now proceed as follows. By the induction hypothesis we know that each of the reduced averages in (\ref{minkowski}) is bounded by $\gamma$ for every pair $N,N' \in [M_{\ast}^{(i)},F_1(M_{\ast}^{(i)})]$ and some $1 \le i \le K$, which depends on the particular average. By the pigeonhole principle and (\ref{contribution}), this implies that we may find some $1 \le i_1 \le K$ such that the contribution to (\ref{minkowski}) of those averages which are not bounded by $\gamma$ for every pair $N,N' \in [M_{\ast}^{(i_1)},F_1(M_{\ast}^{(i_1)})]$ is at most $\left( \frac{K-1}{K} \right)C^{\ast}$. We now apply the induction hypothesis to these remaining averages with the function $F_2$, the parameter $\gamma$ and the starting point $M_{\ast}^{(i_1)}$. This way, for each of these remaining averages, we know that there exists some $1 \le i \le K$ such that the average is bounded by $\gamma$ for every pair $N,N' \in [M_{\ast}^{(i_1,i)},F_2(M_{\ast}^{(i_1,i)})]$. Since by construction of $F_1$ it is $[M_{\ast}^{(i_1,i)},F_2(M_{\ast}^{(i_1,i)})] \subseteq [M_{\ast}^{(i_1)},F_1(M_{\ast}^{(i_1)})]$, we see that those averages which we bounded in the previous step remain bounded by $\gamma$ on each of these new intervals. Thus, we may apply the pigeonhole principle as before to find some $1 \le i_2 \le K$ such that the contribution to (\ref{minkowski}) of those averages which are not bounded by $\gamma$ for every pair $N,N' \in [M_{\ast}^{(i_1,i_2)}, F_2(M_{\ast}^{(i_1,i_2)})]$ is at most $\left( \frac{K-1}{K} \right)^2 C^{\ast}$.

\medskip
Iterating the above process $r$ times, we find a tuple $1 \le i_1, \ldots, i_r \le K$ such that the set of reduced averages which are not bounded by $\gamma$ for every pair $N,N' \in [M_{\ast}^{(i_1, \ldots, i_r)}, F_r(M_{\ast}^{(i_1, \ldots, i_r)})]=[M_{\ast}^{(i_1, \ldots, i_r)}, F(M_{\ast}^{(i_1, \ldots, i_r)})]$, contributes at most 
$$\left( \frac{K-1}{K} \right)^r C^{\ast}< \varepsilon/16,$$
to (\ref{minkowski}), upon choosing $r$ sufficiently large in terms of $\varepsilon$ and $d$. Since the sum over the remaining terms will be bounded by $\sum_{t=0}^{k-1} |\lambda_t| \gamma < \varepsilon/16$, we conclude that (\ref{minkowski}), and therefore $\left\|\mathcal{A}_{N,N'}^{\bf g}[f_1, \ldots, f_{j-1}, f] \right\|_{2}$, is bounded by $\varepsilon/4$, for every $N,N' \in [M_{\ast}^{(i_1,\ldots,i_r)}, F(M_{\ast}^{(i_1, \ldots, i_r)})]$. Notice that while the specific integer $M_{\ast}^{(i_1,\ldots,i_r)}$ we have obtained depends on the set of functions $f_1, \ldots, f_{j-1}, f$ and the system ${\bf g}$, this integer belongs to the sequence $M_{\ast}^{(j_1,\ldots,j_r)}$, $1 \le j_1, \ldots, j_r \le K$, which depends only on $F,\varepsilon,d$ and $M_{\ast}$. The result follows from this observation with the sequence (\ref{sequence}) given by the integers $M_{\ast}^{(j_1, \ldots, j_r)}$, $1 \le j_1, \ldots, j_r \le K$.
\end{proof}

\subsection{Proof of Theorem 3.2}

We can now conclude the proof of Theorem \ref{principal}. As it was done before, we fix $X,G,F, \varepsilon,d$ and ${\bf g}$ as in the statement of that theorem and assume without lost of generality that each reduction ${\bf g}_m^{\ast}$ of ${\bf g}$ is of complexity at most $d-1$ and that the result is already proven for every $d'<d$. We will also write $M_0$ for the integer $M$ to be chosen as the starting point of the sequence (\ref{sequence0}) in Theorem \ref{principal}. Let $\delta, \eta$ be as specified at the beginning of \S \ref{reducible} and write $C_i := C_i^{\delta, \eta}$ for the constants defined in (\ref{central sequence}). Given some positive integer $L$ write $\Sigma_L$ for the set of $L$-reducible functions and set
$$ \Sigma_L^{+} := \Sigma_L \cup B_2(\delta/C^{\ast}),$$
where we write $B_2(\delta/C^{\ast})$ for the set of $f \in L^2(X)$ with $\left\| f \right\|_2 \le \delta/C^{\ast}$. Consider on $L^2(X)$ the norms $\left\| \cdot \right\|_L := \left\| \cdot \right\|_{\Sigma_L^{+}}$ defined as in ($\ref{estructura}$). It is easy to see that these norms are well defined and equivalent to $\left\| \cdot \right\|_{L^2(X)}$ (by the presence of the small $L^2$ ball and the fact that reducible functions are bounded by $1$). Also, notice that $\Sigma_{L+1}^{+} \subseteq \Sigma_{L}^{+}$ for every $L$ which in turn implies (by Lemma \ref{normas}) that $\left\| \cdot \right\|_{L+1}^{\ast} \le \left\| \cdot \right\|_{L}^{\ast}$.

\medskip
Given any integer $M$ write $\psi(M):= F(M^{\ast})$ where $M^{\ast}$ is the integer obtained in Proposition \ref{low dim} with $M_{\ast}=M$. Let $f_1, \ldots, f_{j} \in L^{\infty}(X)$, $\left\| f_i \right\|_{\infty} \le 1$, be given and consider for $f_{j}$ a decomposition of the form provided in Proposition \ref{decomposition} with $\left( \left\| \cdot \right\|_{L} \right)_{L \in \mathbb{N}}$, $\psi, \delta, \eta$ as above and $c$ equal to the constant $c_1$ in Lemma \ref{inverso}. This allows us to find a constant $1 \le C_i \le C^{\ast}$ and some integer $M$ with $M_0 \le M =O_{M_0, \varepsilon, F,d}(1)$, such that
\begin{equation}
\label{partes}
f_{j} = \sum_{t=0}^{k-1} \lambda_t \sigma_t + u + v,
\end{equation}
where $\sum_{t=0}^{k-1} |\lambda_t| \le C_i$, each $\sigma_t$ belongs to $\Sigma_B^+$ for some $B \ge \psi(M)$ (and therefore to $\Sigma_{\psi(M)}^+$), $\left\| u \right\|_{A}^{\ast} < \eta(C_i)$ for some $A < c_1M$ and $\left\| v \right\|_2 < \delta$. We remark that this constant $C_i$ is the one defined in (\ref{central sequence}) and that the integer $M$ obtained belongs to the sequence given in Proposition \ref{decomposition}, which does not depend on the family of norms $\left( \left\| \cdot \right\|_L \right)_{L \in \mathbb{N}}$ and in the present case is therefore independent of the particular system ${\bf g}$ we have fixed (although it certainly depends on its complexity $d$, as well as on $\varepsilon,F$ and $M_0$). Since $\left\| \sum_{t}^{\ast} \lambda_t \sigma_t \right\|_2 \le \delta$, where the sum is restricted to those $\sigma_t \in B_2(\delta/C^{\ast})$, we may assume that each $\sigma_t$ in (\ref{partes}) actually belongs to $\Sigma_{\psi(M)}$, at the cost of softening our bound on $v$ to $\left\| v \right\|_2 < 2\delta$.

\medskip
We would like to use Lemma \ref{inverso} to study the function $u$, but first we need to gain some control on its $L^{\infty}$ norm. In order to do this, denote by $S \in \mathcal{X}$ the set of points on which the inequality $|v(s)| \le C_i$ holds (in particular one has $\mu(S^c) < (2\delta/C_i)^2$) and write $v' := u {\bf 1}_{S^c} + v$. From the fact that $\left\| \sigma_j \right\|_{L^{\infty}(X)} \le 1$ for every $\sigma_j \in \Sigma_{\psi(M)}$, (\ref{partes}) and the definition of $S$, one easily checks that $ | u {\bf 1}_{S^c}(x)| \le 3 | v(x) |$ a.e. and therefore $ \left\| u {\bf 1}_{S^c} \right\|_2 \le 3 \left\| v \right\|_2$. Hence, it follows that for every pair of integers $N,N'$, 
\begin{align}
\label{1}
\nonumber \left\| \mathcal{A}_{N,N'}[f_1, \ldots, f_{j-1}, v'] \right\|_2 &\le \left\|  \mathcal{A}_{N'}[f_1, \ldots, f_{j-1}, v'] \right\|_2+\left\|  \mathcal{A}_{N}[f_1, \ldots, f_{j-1}, v'] \right\|_2 \\
\nonumber &\le 2( 4 \left\| v \right\|_2) \\
&< \varepsilon/3,
\end{align}
where we are using Minkowski's inequality and the fact that $\left\| f_i \right\|_{\infty} \le 1$ for every $1 \le i \le j-1$. Consider now $u {\bf 1}_S$. Similarly as above, one sees that $\left\| u {\bf 1}_S \right\|_{L^{\infty}(X)} \le 3C_i$. Also, it follows from Lemma \ref{normas} that for every $\sigma \in \Sigma_A$ it is
\begin{align*}
\left| \langle u {\bf_1}_S , \sigma \rangle \right| &\le \left| \langle u , \sigma \rangle \right|+\left| \langle u  {\bf_1}_{S^c} , \sigma {\bf_1}_{S^c} \rangle \right| \\
&\le \left\| u \right\|_{A}^{\ast} + \left\| u {\bf 1}_{S^c} \right\|_2 \left\| \sigma {\bf 1}_{S^c} \right\|_2 \\
&< \eta(C_i)+12 \delta^2 /C_i \\
&< 2 \eta(C_i).
\end{align*} 
We are now in a position to apply Lemma \ref{inverso}, which implies that for every pair $N, N' \ge M$
\begin{align}
\label{2}
\nonumber \left\| \mathcal{A}_{N,N'} [f_1, \ldots, f_{j-1},u{\bf 1}_S] \right\|_2 &\le \left\| \mathcal{A}_{N'}[f_1, \ldots, f_{j-1},u{\bf 1}_S] \right\|_2+\left\| \mathcal{A}_{N}[f_1, \ldots, f_{j-1},u{\bf 1}_S] \right\|_2 \\
&\le \varepsilon/3.
\end{align}
It only remains to analyze $\sum_{t=0}^{k-1} \lambda_t \sigma_t$. But we may now invoke Proposition \ref{low dim} to conclude from our choice of $\psi$ that
\begin{equation}
\label{3}
\left\| \mathcal{A}_{N,N'}[f_1, \ldots, f_{j-1}, \sum_{t=0}^{k-1} \lambda_t \sigma_t] \right\|_{L^2(X)} < \varepsilon / 3,
\end{equation}
for every pair $M_i \le N,N' \le F(M_i)$ and some $M_i \in [M,\psi(M)]$ which belongs to the corresponding sequence (\ref{sequence}). Theorem \ref{principal} then follows from (\ref{partes}), (\ref{1}), (\ref{2}), (\ref{3}) and Minkowski's inequality.

\section{The complexity of polynomial systems}
\label{nilpotent}

In this section we will prove that every system of the form given in Theorem \ref{original} has finite complexity, thereby finishing the proof of that theorem. In order to do this, we begin by reviewing some facts about polynomial sequences in nilpotent groups. For a detailed treatment of this topic, the reader is referred to the work of Leibman \cite{L0,L1}.

\medskip
For a $G$-sequence $g=\left(g(n) \right)_{n \in \mathbb{Z}}$ taking values in a nilpotent group $G$ and some integer $m$, we define the operator $D_m$ which takes $g$ to the $G$-sequence $(D_m g)(n):= g(n)g(n+m)^{-1}$. In particular, we have $\langle g|h \rangle_m (n) = (D_m g) (n) h(n+m)$, for every pair of $G$-sequences $g,h$ and every positive integer $m$. We say that a $G$-sequence is {\sl polynomial} if there exists some positive integer $d$ such that for every choice of integers $m_1, \ldots, m_d$, we have $ D_{m_1} \ldots D_{m_d} g  = 1_G$, where we recall that $1_G$ stands for the constant sequence which equals the identity of $G$. It is known that if $\left(g(n) \right)_{n \in \mathbb{Z}}$ is a sequence in a nilpotent group $G$ which is of the form
\begin{equation}
\label{characterization}
 g(n) = T_1^{p_1(n)} \ldots T_k^{p_k(n)},
 \end{equation}
for some $T_1, \ldots, T_k \in G$ and some set of integer valued polynomials $p_1, \ldots, p_k$, then $g$ is a polynomial sequence. Indeed, each $T_i^{p_i(n)}$ is clearly a polynomial sequence and the product of polynomial sequences is polynomial by Lemma \ref{Leibman} below (the converse also holds, see for example \cite{L0}).

\medskip
By a polynomial system we shall mean a system ${\bf g} = (g_1, \ldots, g_j)$, where each $g_i$, $1 \le i \le j$, is a polynomial sequence. We define the size of such a system to be $|{\bf g}|=j$. To prove Theorem \ref{original} it will suffice, by Theorem \ref{principal} and the fact that sequences of the form (\ref{characterization}) are polynomial, to prove that every polynomial system has finite complexity.

\medskip
In order to proceed, we will need to define the degree of a polynomial sequence. Unfortunately, the natural choice of taking the least positive integer $d$ for which every $d$ successive application of the above operators returns the identity is not appropriate for our purposes, since with this definition the set of polynomial sequences of degree $\le d$ need not form a group. In order to amend this, we need to introduce some notation. Write $\mathbb{N}_0 = \mathbb{N} \cup \left\{ 0 \right\}$ and $\mathbb{N}_{\ast} = \mathbb{N}_0 \cup \left\{ -\infty \right\}$. We say a vector $\overline{d} = (d_1, \ldots, d_k) \in \mathbb{N}_{\ast}^{k}$ is {\sl superadditive} if $d_i \le d_{i+1}$ for every $1 \le i <k$ and $d_i + d_j \le d_{i+j}$ for every pair $i,j$, where we are using the conventions $-\infty + t = -\infty$ for every $t \in \mathbb{N}_{\ast}$ and $-\infty < r$ for every $r \in \mathbb{N}_0$. Also, given a superadditive vector $\overline{d} = (d_1, \ldots, d_k)$ and some nonnegative integer $t$, we write $\overline{d}-t := (d_1', \ldots, d_k')$, where $d_i' = d_i-t$ if $t \le d_i$ and $d_i'=-\infty$ otherwise. Notice that $\overline{d}-t$ so defined is also a superadditive vector.

\medskip
Fix a nilpotent group $G$ of nilpotency class $s$ and let $G=G_1 \supset G_2 \supset \ldots \supset G_s \supset G_{s+1}=\left\{ 1_G \right\}$ be its lower central series. As in \cite{L0,L1}, we say a sequence $g=\left(g(n) \right)_{n \in \mathbb{Z}}$ taking values in $G$ is a polynomial sequence of (vector) degree $\le (d_1, \ldots, d_s)$ if $( D_{m_1} \ldots D_{m_{d_k+1}} g) (n) \in G_{k+1}$ for every $n$, every $1 \le k \le s$ and every choice of $m_1, \ldots, m_{d_{k}+1} \in \mathbb{Z}$. If $d_k = - \infty$ we take this to mean that $g$ itself takes values in $G_{k+1}$. We will make use of the following results of Leibman.

\begin{lema}[{\cite[\S 3]{L1}}]
\label{Leibman}
Let $\overline{d}=(d_1,\ldots,d_s)$ be a superadditive vector and let $t,t_1,t_2 \ge 0$ be nonnegative integers. Then we have the following properties.
\begin{enumerate}
	\item If $g$ is a polynomial sequence of degree $\le \overline{d}-t$, then $D_m g$ is a polynomial sequence of degree $\le \overline{d}-(t+1)$, for every $m \in \mathbb{Z}$.
	\item The set of polynomial sequences of degree $\le \overline{d}-t$ forms a group.
	\item If $g$ is a polynomial sequence of degree $\le \overline{d}-t_1$ and $h$ is a polynomial sequence of degree $\le \overline{d}-t_2$, then $[g,h]$ is a polynomial sequence of degree $\le \overline{d}-(t_1+t_2)$, where $[g,h](n):= g^{-1}(n)h^{-1}(n)g(n)h(n)$.
\end{enumerate}
\end{lema}

{\noindent \sl Remark.} The results of \cite{L1} concern the operators $(\widetilde{D}_m g) (n):= g(n)^{-1}g(n+m)=(D_m g^{-1})(n)$. Nevertheless, using Lemma \ref{Leibman} for these operators and a straightforward descending induction on $t$ one can easily check that a $G$-sequence $g$ has degree $\le \overline{d}-t$ with respect to the operators $\widetilde{D}_m$ if and only if it has degree $\le \overline{d}-t$ with respect to the operators $D_m$, from where we recover Lemma \ref{Leibman} as stated.

\medskip
We say a polynomial system ${\bf g}=(g_1, \ldots, g_j)$ has degree $\le d$ if the degree of $g_i$ is $\le d$, for every $1 \le i \le j$. We will show that any system of degree $\le \overline{d}$, for some superadditive vector $\overline{d}=(d_1,\ldots,d_s)$, has finite complexity. Notice that this is enough to prove Theorem \ref{original}, since if a polynomial sequence $g$ has degree $\le (d_1, \ldots, d_s)$, then it also has degree $\le (d,2d,\ldots,sd)$, with $d=\max \left\{ d_i : 1 \le i \le s \right\}$, and this last vector is clearly superadditive.

\medskip
Given a polynomial system ${\bf g}$, we are concerned with the process that consists of passing from ${\bf g}$ to an equivalent system ${\bf g}'$, then taking the $m$-reduction $({\bf g}')_m^{\ast}$ of ${\bf g}'$ for some $m$, passing to an equivalent system $(({\bf g}')_m^{\ast})'$ and then taking the $m'$-reduction of this for some $m'$, etc. What we are free to choose in the above process is to which equivalent system we apply the reductions (but not the integer on which we subsequently reduce) and our objective is to show that there exists some constant $C$, depending on ${\bf g}$, such that for every sequence of positive integers $m,m',m'',\ldots$ we can go to the trivial system $(1_G)$ by means of at most $C$ repetitions of the above transformations. This clearly implies that the complexity of ${\bf g}$ is at most $C$.

\medskip
In order to simplify notation we will omit the reference to the specific sequence of integers on which we reduce. So for instance, we will generically refer to the reduction of a system ${\bf g}=(g_1,\ldots,g_j)$ to be the system
$${\bf g}^{\ast}=(g_1,\ldots,g_{j-1},\langle g_j | 1_G \rangle, \langle g_j | g_1 \rangle, \ldots, \langle g_j | g_{j-1} \rangle ).$$
Similarly, we have the identity 
$$\langle g | h \rangle (n) = Dg (n) (Dh(n))^{-1}h(n),$$
provided, of course, that the omitted subindices are the same. We define a {\sl step} to be the process of passing from a system ${\bf g}$ to the reduction $({\bf g}')^{\ast}$ of some system ${\bf g}'$ equivalent to ${\bf g}$. We will show that one can pass from a polynomial system ${\bf g}$ to the trivial system in a number of steps which is bounded in terms of the size and degree of ${\bf g}$.

\medskip
We define the complete reduction of a system ${\bf g}$ to be the system
$$ {\bf g}^{\ast \ast} = (g_1, \ldots, g_{j-1}, \langle g_j | g_1 \rangle, \ldots, \langle g_j | g_{j-1} \rangle ).$$
Thus, ${\bf g}^{\ast \ast} = {\bf g}^{\ast} \setminus \left\{ \langle g_j | 1_G \rangle \right\}$. We define a {\sl complete step} in the same way as a step, but with the reduction replaced by the complete reduction. Complete steps are needed for a technical reason related to the inductive process to be applied. Precisely, in order to handle steps involving systems of degree $\le \overline{d}$ we will need to assume some control on both steps and complete steps over systems of degree $\le \overline{d}-1$. 

\medskip
Theorem \ref{original} follows from Theorem \ref{principal} and the following result.

\begin{teo}
\label{final}
Let ${\bf g}$ be a polynomial system of size $|{\bf g}| \le C_1$ and degree $\le \overline{d}$, for some superadditive vector $\overline{d}=(d_1,\ldots,d_s)$. Then,
\begin{itemize}
	\item one can go from ${\bf g}$ to the trivial system $(1_G)$ in $O_{C_1,\overline{d}}(1)$ steps,
	\item one can go from ${\bf g}$ to a system consisting of a single sequence of degree $\le \overline{d}$ in $O_{C_1,\overline{d}}(1)$ complete steps,
\end{itemize}
for every sequence of positive integers $m,m',m'',\ldots$. In particular, ${\bf g}$ has complexity $O_{C_1,\overline{d}}(1)$.
\end{teo}

\begin{proof}
Let $\overline{d}$ be as in the statement. We begin by noticing that the result is trivially true for systems of degree $\le \overline{d}-(d_s+1)=(-\infty, \ldots, -\infty)$, since $1_G$ is the only sequence lying in $G_{k+1}$ for every $1 \le k \le s$. We will proceed by induction. Since $\overline{d}-t$ is superadditive for every $t \ge 0$, it will suffice to prove that if Theorem \ref{final} holds for systems of degree $\le \overline{d}-1$ then it also holds for systems of degree $\le \overline{d}$.

\medskip
Thus, let ${\bf g}$ be as in the statement. We will first prove that we can go from ${\bf g}$ to the trivial system in $O_{C_1,\overline{d}}(1)$ steps (and therefore, that ${\bf g}$ has complexity $O_{C_1,\overline{d}}(1)$). In order to do this, observe that ${\bf g}$ can be rewritten in the form
\begin{equation}
\label{sistema}
 {\bf g} = {\bf h}_0 \cup \bigcup_{i=1}^l s_i {\bf h}_i,
 \end{equation}
for some polynomial sequences $s_1, \ldots, s_l$ of degree $\le \overline{d}$, $l \le C_1$, and some polynomial systems ${\bf h}_i$ of degree $\le \overline{d}-1$ and size $\le C_1$, with ${\bf h}_0$ possibly empty (for example, one could simply take $s_i=g_i$ and $h_i=(1_G)$ for every $1 \le i \le l$). Here, if ${\bf h}=(h_1,\ldots,h_k)$, $s{\bf h}$ is the system $(sh_1, \ldots, sh_k)$ and the union of two systems $(h_1, \ldots, h_k)$, $(h_1',\ldots, h_r')$ is understood to be the system $(h_1, \ldots, h_k, h_1', \ldots, h_r')$.

\medskip
The idea will be to show that for systems of the form (\ref{sistema}) one can perform steps in such a way that the resulting systems are also of the form (\ref{sistema}) for the same set of sequences $s_1,\ldots,s_l$. Furthermore, we will show that in finitely many steps we may actually discard the sequence $s_l$, therefore arriving at a system like (\ref{sistema}) in which only the sequences $s_1,\ldots, s_{l-1}$ are present. Iterating this $l$ times we shall then end up with a system of degree $\le \overline{d}-1$, from where one can proceed by induction.

\medskip
In order to carry out this plan we begin by observing that if $s_i,s_j$ are sequences of degree $\le {\overline{d}}$ and $h_i,h_j$ are sequences of degree $\le \overline{d}-1$, we have
\begin{align}
\label{big formula}
\nonumber \langle s_jh_j | s_i h_i \rangle &= D(s_jh_j) (D (s_i h_i))^{-1} s_ih_i \\
\nonumber &= s_i D(s_jh_j) (D (s_i h_i))^{-1} \left[D(s_jh_j) (D (s_i h_i))^{-1},  s_i \right] h_i \\
 &= s_i h^{j,i},
 \end{align}
for some polynomial sequence $h^{j,i}$ which is seen to have degree $\le \overline{d}-1$ by Lemma \ref{Leibman}. Furthermore, if $s_i=s_j=s$, it is easy to check that 
\begin{equation*}
\langle sh_j | sh_i \rangle = s \langle h_j|h_i \rangle.
\end{equation*}
It follows from these formulas that, provided $|{\bf h}_l|>1$, the reduction ${\bf g}^{\ast}$ of ${\bf g}$ is equivalent to a system of the form
\begin{equation}
\label{big reduction}
 {\bf h}_0^{(1)} \cup \left( \bigcup_{i=1}^{l-1} s_i {\bf h}_i^{(1)} \right) \cup s_l {\bf h}_l^{\ast \ast},
 \end{equation}
for some systems ${\bf h}_0^{(1)}, {\bf h}_1^{(1)},\ldots,{\bf h}_{l-1}^{(1)}$ of degree $\le \overline{d}-1$ and size $|{\bf h}_0^{(1)}| \le 2|{\bf h}_0|+1$ and $|{\bf h}_i^{(1)}| \le 2|{\bf h}_i|$ for every other $i$, and where we recall that ${\bf h}_l^{\ast \ast}$ refers to the complete reduction of ${\bf h}_l$. Explicitly, if ${\bf h}_i = (h_{i,1},\ldots,h_{i,j_i})$ for every $0 \le i \le l$, then
$${\bf h}_0^{(1)} = \left(\langle s_l h_{l,j_l} | 1_G \rangle, h_{0,1},\ldots,h_{0,j_0},\langle s_l h_{l,j_l} | h_{0,1} \rangle, \ldots, \langle s_l h_{l,j_l} | h_{0,j_0} \rangle \right),$$
while $s_i{\bf h}_i^{(1)}$ equals
$$\left( s_ih_{i,1},\ldots,s_ih_{i,j_i},\langle s_l h_{l,j_l} | s_ih_{i,1} \rangle, \ldots, \langle s_l h_{l,j_l} | s_ih_{i,j_i} \rangle \right),$$
for every $1 \le i \le l-1$. We see by (\ref{big formula}) that this is of the desired form.

\medskip
Observe now that if ${\bf h}$ is equivalent to ${\bf h}'$ then the system $s{\bf h}$ is also equivalent to $s{\bf h}'$. Since by the induction hypothesis we know that one can pass from ${\bf h}_l$ to a system ${\bf h}$ consisting of a single sequence of degree $\le \overline{d}-1$ in $O_{C_1,\overline{d}-1}(1)$ complete steps, it follows from the above observation and (\ref{big reduction}) that we may pass from ${\bf g}$ to a system of the form
\begin{equation}
\label{casi}
 {\bf h}_0^{(2)} \cup \left( \bigcup_{i=1}^{l-1} s_i {\bf h}_i^{(2)} \right) \cup s_l {\bf h},
 \end{equation}
in $O_{C_1,\overline{d}-1}(1)$ steps, where each system ${\bf h}_i^{(2)}$ has degree $\le \overline{d}-1$ and size $O_{C_1,\overline{d}}(1)$, and ${\bf h}$ is a system consisting of a single sequence of degree $\le \overline{d}-1$. But then we see from (\ref{big formula}) that the reduction of (\ref{casi}) will be of the form
$$  {\bf h}_0^{(3)} \cup \left( \bigcup_{i=1}^{l-1} s_i {\bf h}_i^{(3)} \right),$$
with each ${\bf h}_i^{(3)}$ having degree $\le \overline{d}-1$ and size $O_{C_1,\overline{d}}(1)$. We have therefore succeeded in discarding the sequence $s_l$ from our system. We can now repeat the same process as before with $s_{l-1}$ in place of $s_l$. Since the size of ${\bf h}_{l-1}^{(3)}$ is $O_{C_1,\overline{d}}(1)$, we see that this new process finishes in $O_{C_1,\overline{d}}(1)$ steps, leaving us with a system of the form
$$  {\bf h}_0^{(4)} \cup \left( \bigcup_{i=1}^{l-2} s_i {\bf h}_i^{(4)} \right).$$
Therefore, iterating the above process $l$ times, we are finally left in $O_{C_1,\overline{d}}(1)$ steps with a system of degree $\le \overline{d}-1$ from where we may apply the induction hypothesis to obtain the trivial system in $O_{C_1,\overline{d}}(1)$ further steps, thereby completing the proof of the finite complexity of ${\bf g}$.

\medskip
Now it only remains to show that one can pass from ${\bf g}$ to a system consisting of a single sequence of degree $\le \overline{d}$ in $O_{C_1,\overline{d}}(1)$ complete steps. But it is clear that the above reasoning to pass from ${\bf g}$ to a system of degree $\le \overline{d}-1$ works in exactly the same way for complete steps, since the only things that may change are the systems ${\bf h}_0^{(1)},{\bf h}_0^{(2)}, {\bf h}_0^{(3)}, \ldots$, which nevertheless will always have degree $\le \overline{d}-1$ and whose size may only be smaller than in the previous case. Thus, the above reasoning allows us to pass to a system of degree $\le \overline{d}-1$ from where we may apply induction, as long as we are not left after any of the complete steps with a system which can be written in its entirety as $s_i{\bf h}$, for some $s_i$ as above and ${\bf h}$ of degree $\le \overline{d}-1$ and size $|{\bf h}|=1$ (because if the whole system has size $1$ the complete reduction is not defined). But since in such a case we are already done, this completes the proof of Theorem \ref{final} and therefore of Theorem \ref{original}.
\end{proof}

\section{Further results}
\label{further}

The next result is easily seen to follow from the methods of this paper.

\begin{teo}
\label{general}
Let $G$ be a nilpotent group of measure preserving transformations of a probability space $(X, \mathcal{X}, \mu)$. Then, for every $T_1, \ldots, T_l \in G$, every $f_1,\ldots, f_r \in L^{\infty}(X)$, every set of polynomials $p_{i,j}:\mathbb{Z}^d \rightarrow \mathbb{Z}$, and every F{\o}lner sequence $\left\{ \Phi_N \right\}_{N=1}^{\infty}$ in $\mathbb{Z}^d$, the averages
\begin{equation}
\label{general ergodic averages}
\frac{1}{|\Phi_N|} \sum_{u \in \Phi_N} \prod_{j=1}^r  \left( T_1^{p_{1,j}(u)} \ldots T_l^{p_{l,j}(u)} \right) f_j
\end{equation}
converge in $L^2(X, \mathcal{X}, \mu)$.
\end{teo}

During the proof of Theorem \ref{original} we used crucially the fact that the $L^{\infty}$ norm is an algebra norm ($\left\| fg \right\|_{\infty} \le \left\| f \right\|_{\infty} \left\| g \right\|_{\infty}$). While this is clearly not true for the $L^2$ norm, if we are concerned with the study of a single function $f \in L^2(X)$, this issue is no longer present. Furthermore, in this case our polynomial systems will always have size $1$, a fact that allows us to drop the hypothesis of nilpotency on our group $G$ (because we no longer need the product of polynomial sequences to be polynomial). More generally, it is easy to see from these observations that our methods produce the following result, which was also conjectured by Bergelson and Leibman in \cite{BL}.

\begin{teo}
Let $G$ be a group of unitary operators on a Hilbert space $\mathcal{H}$. If $(g(n))_{n \in \mathbb{Z}}$ is a polynomial sequence in $G$, then
$$ \lim_{N \rightarrow \infty} \frac{1}{N} \sum_{n=1}^N g(n) u$$
exists for every $u \in \mathcal{H}$.
\end{teo}

If the group $G$ is assumed to be nilpotent this provides an alternative proof of a result already established by Bergelson and Leibman in \cite{BL}. Also, it is important to notice that it is not presently known whether if by dropping the nilpotency assumption on the group one is in fact obtaining a more general result.

\appendix
\section{Some examples of reductions}

We now provide some concrete examples of how the process studied in Section \ref{nilpotent} returns the trivial system for some polynomial systems. Given systems ${\bf g}$ and ${\bf h}$ we write ${\bf g} \sim {\bf h}$ to mean that both systems are equivalent and we write ${\bf g} \stackrel{m}{\rightarrow} {\bf h}$ to mean that ${\bf h}$ is the $m$-reduction of ${\bf g}$. Given measure preserving transformations $T,S,C$ of a probability space, we will use the convention of writing $T^{n^2}S^nC$ for the $G$-sequence $g$ given by $g(n)=T^{n^2}S^n C$.

\medskip
Before giving the examples, we note that an unpleasant feature of the process described in the body of the paper is that in the simplest cases it is unnecessarily complicated (mainly, this happens when the group is abelian). A way to amend this is the trivial observation that the averages (\ref{system averages}) for sequences $g_i(n)=h_i(n)C_i$, with $C_i$ some set of transformations, equal the averages associated to the system ${\bf h}=(h_1,\ldots,h_j)$ evaluated at the functions $C_i f_i$. Thus, we may extend the previous equivalence relation to include those pairs of systems which may be obtained from each other by adding or removing a constant in the above manner. It is easy to check that the arguments of \S \ref{main} work equally well with this alternative notion of equivalence. We write ${\bf g} \sim^{\ast} {\bf h}$ if ${\bf g}$ and ${\bf h}$ are equivalent in this way and call this modification of the process 'cheating'. While there is no substantial gain by using this slight modification in the general case, many of the examples discussed below become much cleaner in this way. Nevertheless, we will also show in every case how the process is performed without cheating.

\begin{ej}
\label{constant}
As a trivial example, suppose our system is constant i.e. of the form $(C_1,\ldots,C_j)$ for some constant $G$-sequences $C_1, \ldots, C_j$. Then its $m$-reduction is equivalent to $(1_G,C_1,\ldots,C_{j-1})$ for every $m$, so in particular we get the trivial system after at most $j$ steps. Of course, if one is allowed to cheat, one has $(C_1,\ldots,C_j)\sim^{\ast} (1_G)$ to begin with.
\end{ej}

\begin{ej} 
\label{lineal}
Suppose we are given a linear system $(L_1^n C_1,\ldots,L_j^n C_j)$ consisting of commuting transformations $L_1,C_1, \ldots,L_j,C_j$. The $m$-reduction of this system is given by
\begin{align}
\label{1}
(L_1^{n}C_1, \ldots,L_{j-1}^nC_{j-1},L_j^{-m},& L_j^{-m} L_1^{n+m} C_1, \ldots, L_j^{-m} L_{j-1}^{n+m}C_{j-1}) \\
\nonumber & \sim^{ \ast} (L_1^n, \ldots, L_{j-1}^n).
\end{align}
This highlights the advantage of cheating. Indeed, this would allow us to go to the trivial system in $j$ steps, while without cheating we would require more than $2 \uparrow \uparrow j$ steps. We now see how the latter is accomplished. Our objective is to eliminate $L_{j-1}^n$ from the reduction (\ref{1}) (cf. the general strategy given in the proof of Theorem \ref{final}). In general, if we are given a system of the form
$$ (L_1^n C_{1,1}, \ldots, L_1^n C_{1,i_1}, \ldots, L_k^n C_{k,1}, \ldots, L_k^n C_{k,i_k} ),$$
with the transformations generating an abelian group, its $m$-reduction will be equivalent to
\begin{align*}
&(L_k^{-m}, L_1^n C_{1,1}, \ldots, L_1^n C_{1,i_1}, L_1^n C_{1,1}L_k^{-m}L_1^m, \ldots,L_1^n C_{1,i_1}L_k^{-m}L_1^m,\ldots,\\
&L_{k-1}^n C_{k-1,1},\ldots, L_{k-1}^n C_{k-1,i_{k-1}}, L_{k-1}^n C_{k-1,1}L_k^{-m}L_{k-1}^m, \ldots, L_{k-1}^n C_{k-1,i_{k-1}}L_k^{-m}L_{k-1}^m,\\
&L_k^n C_{k,1}, \ldots, L_k^n C_{k,i_k-1} ).
 \end{align*}
In particular, $L_i^n$ appears twice as many times as before for every $1 \le i < k$, while $L_k^n$ appears one time less, therefore disappearing after $i_k$ steps. Notice also that at each step we get twice plus one as many constant sequences as before. Applying these observations we see that the system $(L_1^n C_1,\ldots,L_j^n C_j)$ reduces to one consisting only of constant sequences after at most $a(j)$ steps, with $a:\mathbb{N} \rightarrow \mathbb{N}$ the function recursively defined by $a(1)=1$ and $a(n+1)=a(n)+2^{a(n)}$. We may then proceed as in Example \ref{constant}.
\end{ej}

\begin{ej}
\label{1 cuadrado}
Consider now a system of the form $(S_1^n C_1,\ldots,S_{j-1}^n C_{j-1}, T^{n^2} S_j^n C_j)$ for commuting $T,S_i,C_i$. The $m$-reduction is given by
\begin{align*}
 (S_1^n C_1,& \ldots, S_{j-1}^n C_{j-1}, T^{-2mn-m^2}S_j^{-m}, T^{-2mn-m^2} S_j^{-m} S_1^{n+m}C_1, \ldots,\\
 &T^{-2mn-m^2} S_j^{-m} S_{j-1}^{n+m}C_{j-1}),
\end{align*}
which is a linear system and therefore reduces to the trivial system by the procedure discussed in Example \ref{lineal}.
\end{ej}

\begin{ej}
If we are given a system of size $1$ consisting of a polynomial $G$-sequence $g$ then it is obvious that the number of steps required to reach the trivial system is the total degree of $g$. Notice that this is true even when the group $G$ is not assumed to be nilpotent.
\end{ej}

\begin{ej}
Consider $(T^{n^2}, T^{n^2}S^n)$ for commuting $T$ and $S$. The $m$-reduction is given by
\begin{align*}
(T^{n^2}, T^{-2mn-m^2}&S^{-m} , T^{n^2}S^{-m})\\
&\stackrel{l}{\rightarrow} (T^{n^2}, T^{-2mn-m^2}S^{-m} , T^{-2ln-l^2}, T^{n^2}, T^{-2ln-l^2-2m(n+l)-m^2}S^{-m}),
\end{align*}
and this is equivalent to a system of the form studied in Example \ref{1 cuadrado}.
\end{ej}

\begin{ej}
Consider $(T^n,S^n)$ with $T$ and $S$ generating a nilpotent group. The $m_1$-reduction is given by
$$ (T^n, S^{-m_1},S^{-m_1}T^{m_1}T^n).$$
Write $C := S^{-m_1}T^{m_1}$, $C^{(1)} := S^{-m_1}$. Then the $m_2$-reduction of this is given by
\begin{align*}
(T^n, C^{(1)}, C T^{-m_2} C^{-1},&C T^{-m_2} C^{-1}T^{m_2} T^n, C T^{-m_2} C^{-1} C^{(1)} ) \\
& \sim (C^{(1)},C^{(2)},C^{(3)},T^n, [C^{-1},T^{m_2}]T^n).
\end{align*}
for some constant $G$-sequences $C^{(2)}, C^{(3)}$ which depend on $m_1,m_2$. By the same reasoning we see that after reducing at $m_3, \ldots, m_l$ (and passing to equivalent systems) we get the system
$$ (C^{(1)}, C^{(2)}, \ldots, C^{(c(l))}, T^n, [[[[C^{-1},T^{m_2}]^{-1},T^{m_3}]^{-1},\ldots]^{-1},T^{m_l}]T^n),$$
for some constant $G$-sequences $C^{(i)}$, $1 \le i \le c(l)$, with $c:\mathbb{N} \rightarrow \mathbb{N}$ the increasing function defined recursively by $c(1)=1$ and $c(n+1)=2c(n)+1$. Clearly, this is equivalent to $(C^{(0)}, C^{(1)}, \ldots, C^{(c(l))}, T^n)$ for some $l$ of size at most $s+1$, with $s$ the nilpotency class of the group. Since any reduction of this last system will be a constant system of size $c(l+1)$ it follows that our original system $(S^n,T^n)$ reduces to the trivial one in at most $s+2+c(s+1)$ steps. Of course, $s+2$ steps would have sufficed if we were allowed to cheat.
\end{ej}

\begin{ej}
Our last example is the system $(T^{n^2}, S^{n^2})$ for commuting $T$ and $S$. We have
\begin{align*}
(T^{n^2}, S^{n^2}) &\stackrel{m_1}{\rightarrow} (T^{n^2}, S^{-2nm_1-m_1^2}, S^{-2nm_1-m_1^2} T^{n^2} T^{2nm_1+m_1^2}) \\
&\sim (S^{-2nm_1-m_1^2}, T^{n^2}, S^{-2nm_1-m_1^2} T^{n^2} T^{2nm_1+m_1^2}) \\
 &\stackrel{m_2}{\rightarrow} (S^{-2nm_1-m_1^2},T^{n^2}, T^{-2nm_2-m_2^2-2m_1m_2}S^{2m_1m_2},\\ 
  &T^{-2nm_2-m_2^2-2m_1m_2}S^{-2nm_1-m_1^2},T^{n^2}S^{2m_1m_2}T^{-2m_1m_2}) \displaybreak[0] \\
&\stackrel{m_3}{\rightarrow} (S^{-2nm_1-m_1^2},T^{n^2}, T^{-2nm_2-m_2^2-2m_1m_2}S^{2m_1m_2},\\ 
  &T^{-2nm_2-m_2^2-2m_1m_2}S^{-2nm_1-m_1^2}, T^{-2nm_3-m_3^2},\\
  &T^{-2nm_3-m_3^2}S^{-2nm_1-2m_1m_3-m_1^2}, T^{n^2},\\
  &T^{-2nm_3-m_3^2-2nm_2-2m_2m_3-m_2^2-2m_1m_2}S^{2m_1m_2},\\
  & T^{-2nm_3-m_3^2-2nm_2-2m_2m_3-m_2^2-2m_1m_2}S^{-2nm_1-2m_1m_3-m_1^2})
 \end{align*} 
and this last system is equivalent to one of the form studied in Example \ref{1 cuadrado}.
\end{ej}

\end{document}